\documentclass[a4paper]{amsart}

\usepackage{dsfont,tikz-cd,hyperref}
\newtheorem{maintheorem}{Theorem}
\newtheorem{mainproposition}[maintheorem]{Proposition}

\newtheorem{theorem}{Theorem}[section]
\newtheorem{corollary}[theorem]{Corollary}
\newtheorem{lemma}[theorem]{Lemma}
\newtheorem{proposition}[theorem]{Proposition}

\theoremstyle{definition}
\newtheorem{definition}[theorem]{Definition}
\newtheorem{example}[theorem]{Example}

\numberwithin{equation}{section}

\DeclareMathOperator{\End}{End}
\DeclareMathOperator{\Ext}{Ext}
\DeclareMathOperator{\Hom}{Hom}
\renewcommand{\Im}{\operatorname{Im}}
\DeclareMathOperator{\inc}{inc}
\DeclareMathOperator{\inj}{inj}

\DeclareMathOperator{\Ker}{Ker}
\DeclareMathOperator{\Mod}{Mod}
\DeclareMathOperator{\proj}{proj}
\DeclareMathOperator{\rad}{rad}
\DeclareMathOperator{\rep}{rep}
\DeclareMathOperator{\Rep}{Rep}
\DeclareMathOperator{\soc}{soc}
\DeclareMathOperator{\supp}{supp}
\renewcommand{\top}{\operatorname{top}}
\newcommand{\1}{\mathds{1}}
\newcommand{\N}{\mathbb{N}}
\newcommand{\op}{\mathrm{op}}

\newcommand{\Homc}{\Hom_\mathcal{C}}

\newcommand{\To}{\longrightarrow}

\newcommand{\abs}[1]{\left\lvert#1\right\rvert}
\newcommand{\set}[1]{\left\{#1\right\}}

\title[Projective objects in $\mathrm{rep}(Q)$]
  {Projective objects in the category of pointwise finite dimensional
  representations of an interval finite quiver}

\author{Pengjie Jiao}
\address{Department of Mathematics,
  China Jiliang University,
  Hangzhou 310018, PR China}
\email{jiaopjie@cjlu.edu.cn}

\date{\today}

\begin{document}

% Abstract -------------------------------------------------------
\begin{abstract}
  For an interval finite quiver $Q$, we introduce a class of flat representations.
  We classify the indecomposable projective objects in the category $\mathrm{rep}(Q)$ of pointwise finite dimensional representations.
  We show that an object in $\mathrm{rep}(Q)$ is projective if and only if it is a direct sum of countably generated flat representations.
\end{abstract}

\maketitle
% ----------------------------------------------------------------

\section{Introduction}
\label{sec:intro}

Let $k$ be a field and let $Q$ be an interval finite quiver (may contain infinitely many vertices). Denote by $\Rep(Q)$ the category of representations of $Q$ over $k$, and by $\rep(Q)$ the full subcategory formed by pointwise finite dimensional representations.

We mention that infinite quivers appear naturally in the covering theory of algebras;
see \cite{BongartzGabriel1982Covering,Gabriel1981universal}.
The representation theory of some infinite quivers is studied in
\cite{BautistaLiuPaquette2013Representation}.
The result is applied to the study of the bounded derived category of an algebra with radical square zero;
see \cite{BautistaLiu2017bounded}.
%It suggests that the study of projective objects in $\rep(Q)$ makes sense.

We are interested in the Auslander--Reiten theory of $\rep(Q)$.
%compare \cite{BautistaLiuPaquette2013Representation}.
The reason is due to that almost split sequences in $\rep(Q)$ seem to behave better than the ones in certain subcategories of $\rep(Q)$;
see \cite[Section~2]{BautistaLiuPaquette2013Representation}.
Moreover, $\rep(Q)$ is easier to study than $\Rep(Q)$.

The very first step is to understand the projective objects and then the projectively trivial morphisms (see \cite[Section~2]{LenzingZuazua2004Auslander}) in $\rep(Q)$.
It is well known that each representation in $\rep(Q)$ is a direct sum of indecomposable representations, whose endomorphism ring is local; see Lemma~\ref{lem:dec}.
Hence, it is sufficient to study indecomposable projective objects in $\rep(Q)$.

In this paper, we are able to classify the indecomposable projective objects in $\rep(Q)$, and characterize the projective objects via flat representations of $Q$.
To state the results, we introduce some notations.

For each vertex $a$, we denote by $P_a$ the usual indecomposable projective representation corresponding to $a$.

Let $p$ be a right infinite path, which means an infinite sequence of arrows $\alpha_1 \alpha_2 \cdots \alpha_n \cdots$ with $s(\alpha_i) = t(\alpha_{i+1})$ for each $i \geq 1$.
Recall that the \emph{convex hull} of $p$ is the smallest convex subquiver of $Q$ containing $p$.
It is called \emph{uniformly interval finite} if the set of finite paths $u$ with $s(u)=a, t(u)=b$ for any given vertices $a,b$ is bounded uniformly.

Denote by $[p]$ the equivalence class (see page~\pageref{def:equiv} for the definition) of right infinite paths containing $p$.
We introduce a flat representation $X_{[p]}$ and show that it is an indecomposable projective object in $\rep(Q)$, if the convex hull of any right infinite path in $[p]$ is uniformly interval finite; see Proposition~\ref{prop:Xp-proj}.

Moreover, we give a complete classification of indecomposable projective objects, and hence projective objects in $\rep(Q)$.

\begin{maintheorem}[see Theorem~\ref{thm:classify}]
  Let $Q$ be an interval finite quiver. Assume that $P$ is an indecomposable projective object in $\rep(Q)$. Then either $P \simeq P_{a}$ for some vertex $a$, or $P \simeq X_{[p]}$ for some equivalence class $[p]$ of right infinite paths, where the convex hull of any right infinite path in $[p]$ is uniformly interval finite.
\end{maintheorem}

It is well known that a finitely presented representation is projective if and only if it is flat. Inspired by this fact, we investigate the relationship between flat representations lying in $\rep(Q)$ and projective objects in $\rep(Q)$.

More precisely, we obtain the following characterization for flat representations.
This strengthens the Lazard--Govorov Theorem \cite{Govorov1965flat,Lazard1964Sur} in the special case $\Rep(Q)$.

\begin{mainproposition}[see Proposition~\ref{prop:flat}]
  A flat representation in $\Rep(Q)$ is a direct limit of finitely generated projective subrepresentations.
\end{mainproposition}

Based on this description, we can give a characterization for projective objects in terms of flat representations. This is analogous to the classical result \cite[Theorem~2.2]{Drinfeld2006Infinite} in module categories, which is due to \cite{Kaplansky1958Projective} and \cite{RaynaudGruson1971Criteres}.

\begin{maintheorem}[see Theorem~\ref{thm:flat}]
  Let $Q$ be an interval finite quiver and let $M$ be a representation in $\rep(Q)$. Then $M$ is projective in $\rep(Q)$ if and only if $M$ is a direct sum of countably generated flat representations in $\Rep(Q)$.
\end{maintheorem}

The paper is organized as follows.
In Section~2 we introduce the notion of uniformly interval finite quiver. We give characterizations for the convex hull of a right infinite path being uniformly interval finite.
In Section~3 we recall some basic facts about representations of quivers.
In Section~4 we recall some results about direct limits and inverse limits.
In Section~5 we introduce a class of flat representations $X_{[p]}$ for each equivalence class $[p]$ of right infinite paths. We study morphisms between representations of the form $X_{[p]}$. We give a characterization for flat representations in $\Rep(Q)$.
Sections~6 and 7 are dedicated to the proofs of Theorems~\ref{thm:classify} and \ref{thm:flat}, respectively.

\section{Uniformly interval finite quivers}
\label{sec:quiver}

Let $Q=(Q_0,Q_1)$ be a quiver, where $Q_0$ is the set of vertices and $Q_1$ is the set of arrows.
Given an arrow $\alpha \colon a \to b$, we denote by $s(\alpha) = a$ its source and by $t(\alpha) = b$ its target.
%Given a vertex $a$, we denote by $a^-$ the set of vertices $b$, where there exists an arrow $b \to a$; denote by $a^+$ the set of vertices $b$, where there exists an arrow $a \to b$.

A (finite) path $p$ of length $l\geq1$ is a sequence of arrows $\alpha_l \cdots \alpha_2 \alpha_1$ such that $s(\alpha_{i+1}) = t(\alpha_i)$ for each $i = 1, 2, \dots, l-1$. We set $s(p) = s(\alpha_1)$ and $t(p) = t(\alpha_l)$.
We associate with each vertex $a$ a trivial path (of length 0) $e_a$ such that $s(e_a) = a = t(e_a)$.
A nontrivial finite path $p$ is called an oriented cycle if $s(p) = t(p)$.

Let $a, b \in Q_0$. Denote by $Q(a,b)$ the set of finite paths $p$ with $s(p) = a$ and $t(p) = b$. If $Q(a,b) \neq \emptyset$, then $a$ is called a predecessor of $b$, and $b$ is called a successor of $a$. Moreover, if $Q(a,b)$ contains an arrow, then $a$ is called a direct predecessor of $b$, and $b$ is called a direct successor of $a$. Denote by $a^+$ the set of its direct successors, and by $b^-$ the set of its direct predecessors.

A \emph{right infinite path} $p$ is an infinite sequence of arrows $\alpha_1 \alpha_2 \cdots \alpha_n \cdots$ such that $s(\alpha_i) = t(\alpha_{i+1})$ for each $i \geq 1$. We set $t(p) = t(\alpha_1)$.
Dually, a \emph{left infinite path} $p$ is an infinite sequence of arrows $\cdots \alpha_n \cdots \alpha_2 \alpha_1$ such that $s(\alpha_{i+1}) = t(\alpha_i)$ for each $i \geq 1$. We set $s(p) = s(\alpha_1)$.
Here, we use the terminologies in \cite[Subsection~2.1]{Chen2015Irreducible}.
We mention that the pair of notions are opposite to the corresponding ones in
\cite[Section~1]{BautistaLiuPaquette2013Representation}.

Recall that $Q$ is called \emph{connected} if its underlying graph is connected.
We call $Q$ \emph{interval finite} if $Q(a,b)$ is finite for any vertices $a$ and $b$. We mention that an interval finite quiver contains no oriented cycles.
We call $Q$ \emph{locally finite} if for any vertex $a$, the set of arrows $\alpha$ with $s(\alpha) = a$ or $t(\alpha) = a$ is finite. We call $Q$ \emph{strongly locally finite} for short if $Q$ is locally finite and interval finite.

The opposite quiver $Q^\op$ of $Q$ means the quiver $(Q_0^\op, Q_1^\op)$, where $Q_0^\op = Q_0$ and $Q_1^\op = \set{\alpha^\op \colon b \to a \middle\vert Q_1 \ni \alpha \colon a \to b}$.
Given a finite path $p = \alpha_n \cdots \alpha_2 \alpha_1$ in $Q$, we set the corresponding finite path $p^\op = \alpha_1^\op \alpha_2^\op \cdots \alpha_n^\op$ in $Q^\op$.
Similarly, given a right infinite path $p = \alpha_1 \alpha_2 \cdots \alpha_n \cdots$ and a left infinite path $q = \cdots \beta_n \cdots \beta_2 \beta_1$ in $Q$, we set the corresponding left infinite path $p^\op = \cdots \alpha_n^\op \cdots \alpha_2^\op \alpha_1^\op$ and right infinite path $q^\op = \beta_1^\op \beta_2^\op \cdots \beta_n^\op \cdots$ in $Q^\op$.

Following \cite[Subsection~2.1]{Chen2015Irreducible}, we introduce an equivalence relation on right infinite paths in $Q$.
\label{def:equiv}
Two right infinite paths $\alpha_1 \alpha_2 \cdots \alpha_i \cdots$ and $\beta_1 \beta_2 \cdots \beta_j \cdots$ are equivalent, if there exist some positive integers $m$ and $n$ such that
\[
  \alpha_m \alpha_{m+1} \cdots \alpha_i \cdots
  = \beta_n \beta_{n+1} \cdots \beta_j \cdots.
\]
Given a right infinite path $p$, we denote by $[p]$ the equivalence class containing $p$. For each vertex $a$, we denote by $[p]_a$ the subclass of $[p]$ consisting of right infinite paths $p'$ with $t(p')=a$. We mention that both $[p]$ and $[p]_a$ are sets.

Dually, two left infinite paths $\cdots \alpha_i \cdots \alpha_2 \alpha_1$ and $\cdots \beta_j \cdots \beta_2 \beta_1$ are equivalent, if there exist some positive integers $m$ and $n$ such that
\[
  \cdots \alpha_i \cdots \alpha_{m+1} \alpha_m
  = \cdots \beta_j \cdots \beta_{n+1} \beta_n.
\]
Given a left infinite path $q$, we denote by $[q]$ the equivalence class containing $q$.

Recall that a right infinite path is called \emph{cyclic}, if it is of the form $u u \cdots u \cdots$ for some oriented cycle $u$.
A right infinite path is called \emph{rational} if it is equivalent to a cyclic right infinite path; otherwise, it is called \emph{irrational}.
We observe that a rational right infinite path can be written as the form $q u u \cdots u \cdots$ for some finite path $q$ and some oriented cycle $u$.
Then any right infinite path is irrational if $Q$ contains no oriented cycles.

For a right infinite path $p$, we mention the following observation.

\begin{lemma}\label{lem:cycle}
  Let $p$ be a right infinite path. Assume $qp = q'p$ for some finite paths $q$ and $q'$ with $s(q) = t(p)= s(q')$ and $q \neq q'$. Then $p$ is cyclic.
\end{lemma}

\begin{proof}
  We observe that the lengths of $q$ and $q'$ are not the same, since $qp = q'p$ and $q \neq q'$. We may assume the length of $q'$ is greater than the one of $q$. Then there exists some nontrivial finite path $u$ such that $q' = q u$. We have that $s(u) = s(q') = s(q) = t(u)$. In other words, $u$ is an oriented cycle. We observe that $qp = q'p = qup$, and then $p = up$. It follows inductively that $p = u u \cdots u \cdots$.
\end{proof}

Recall that a subquiver $Q'$ of $Q$ is called \emph{full} if each arrow $\alpha$ with $s(\alpha), t(\alpha) \in Q'_0$ lies in $Q'$. We call $Q'$ \emph{convex} if each finite path $p$ with $s(p),t(p)\in Q'_0$ lies in $Q'$. Given a finite path (or an infinite path) $p$, the smallest convex subquiver of $Q$ containing $p$ is called the \emph{convex hull} of $p$.

%We define inductively full subquivers $_n Q$ of $Q$ with $n \geq -1$ as follows. Let $_{-1} Q$ be the empty quiver. For each $n \geq 0$, the full subquiver $_n Q$ is generated by the vertices $a$ with $a^-\subseteq{_{n-1}Q}$. We mention that $_n Q \subseteq {_{n+1} Q}$ for each $n \geq -1$. We set $_\infty Q = \bigcup_{n \geq -1} {_nQ}$.

Given an equivalence class $[p]$ of irrational right infinite paths, we have the following characterization of $\abs{[p]_{t(p)}}$. Here, the symbol ``$\abs{\cdot}$'' means the cardinal number of a set.

\begin{lemma}\label{lem:[p]_a}
  Let $p = \alpha_1 \alpha_2 \cdots \alpha_i \cdots$ be an irrational right infinite path.
  Denote by $\Omega$ the convex hull of $p$, and set $a_i = t(\alpha_{i+1})$ for any $i \geq0 $.
  Then
  \[
    \abs{[p]_{a_0}}
    = \sup_{a,b\in \Omega} \abs{Q(a,b)}
    = \sup_{i,j\geq 0} \abs{Q(a_i, a_j)}
    = \sup_{i\geq 0} \abs{Q(a_i, a_0)}.
  \]
\end{lemma}

\begin{proof}
  We observe by the definition of the convex hull that for any vertices $a$ and $b$ in $\Omega$, there exist some finite paths $u \in Q(b, a_0)$ and $v \in Q(a_i, a)$ for some $i\geq0$. The injection
  \[
    f \colon Q(a,b) \To Q(a_i, a_0), \quad
    q \mapsto u q v,
  \]
  implies that $\abs{Q(a,b)} \leq \abs{Q(a_i, a_0)}$. We then obtain
  \[
    \sup_{i\geq 0} \abs{Q(a_i, a_0)}
    \geq \sup_{a,b\in \Omega} \abs{Q(a,b)}
    \geq \sup_{i,j\geq 0} \abs{Q(a_i, a_j)}
    \geq \sup_{i\geq 0} \abs{Q(a_i, a_0)}.
  \]
  It remains to show $\abs{[p]_{a_0}} = \sup_{i\geq 0} \abs{Q(a_i, a_0)}$.

  For each $i \geq 0$, let $\Delta_i$ be the subset of $[p]_{a_0}$ consisting of right infinite paths of the form $u \alpha_{i+1} \alpha_{i+2} \cdots \alpha_j \cdots$ for some finite path $u$.
  We have that $\Delta_i \subseteq \Delta_{i+1}$ and $[p]_{a_0} = \bigcup_{i \geq 0} \Delta_i$. Consider the surjection
  \[
    g \colon Q(a_i, a_0) \To \Delta_i, \quad
    q \mapsto q \alpha_{i+1} \alpha_{i+2} \cdots \alpha_j \cdots.
  \]
  By the assumption, the right infinite path $\alpha_{i+1} \alpha_{i+2} \cdots \alpha_j \cdots$ can not be cyclic. Then Lemma~\ref{lem:cycle} implies that $g$ is an injection, and hence is a bijection. It follows that
  \[
    \abs{[p]_{a_0}}
    = \sup_{i \geq 0} \abs{\Delta_i}
    = \sup_{i \geq 0} \abs{Q(a_i, a_0)}.
    \qedhere
  \]
\end{proof}

Here, we introduce a notion stronger than interval finite.

\begin{definition}\label{def:uif}
  We call a quiver \emph{uniformly interval finite} if there exists some integer $N$ such that for each pair of vertices $a$ and $b$, the number of finite paths $p$ with $s(p) = a$ and $t(p) = b$ is less than or equal to $N$.
  \qed
\end{definition}

\begin{example}
  \begin{enumerate}
    \item
      The infinite quivers of $A_\infty$, $A_\infty^\infty$ and $D_\infty$ type are uniformly interval finite.
    \item
      The infinite quiver of the form
      \[\begin{tikzcd}
        \circ
        &\circ  \lar[yshift=.3em] \lar[yshift=-.3em]
        &\circ  \lar[yshift=.3em] \lar[yshift=-.3em]
        &\cdots \lar[yshift=.3em] \lar[yshift=-.3em]
        &\circ  \lar[yshift=.3em] \lar[yshift=-.3em]
        &\cdots \lar[yshift=.3em] \lar[yshift=-.3em]
      \end{tikzcd}\]
      is interval finite but not uniformly interval finite.
    \qed
  \end{enumerate}
\end{example}

The following result is a consequence of Lemma~\ref{lem:[p]_a}.

\begin{proposition}\label{prop:uif}
  Let $p = \alpha_1 \alpha_2 \cdots \alpha_i \cdots$ be an irrational right infinite path.
  Set $a_i = t(\alpha_{i+1})$ for each $i \geq 0 $.
  The following statements are equivalent.
  \begin{enumerate}
    \item \label{item:prop:uif:1}
      The convex hull of $p$ is uniformly interval finite.
    \item \label{item:prop:uif:2}
      $\set{\abs{Q(a_i, a_0)} \,\middle\vert i\geq0}$ is bounded.
    \item \label{item:prop:uif:3}
      $[p]_{a_0}$ is finite.
  \end{enumerate}
\end{proposition}

\begin{proof}
  Let $\Omega$ be the convex hull of $p$. We observe that $\Omega$ being uniformly interval finite precisely means $\set{\abs{Q(a,b)} \,\middle\vert a,b \in \Omega}$ being bounded.
  Then it follows from Lemma~\ref{lem:[p]_a} that (1), (2) and (3) are equivalent.
\end{proof}

The following lemma gives some necessary conditions for the convex hull of a right infinite path being uniformly interval finite.

\begin{lemma}\label{lem:uif}
  Let $p = \alpha_1 \alpha_2 \cdots \alpha_i \cdots$ be a right infinite path, whose convex hull is uniformly interval finite. Set $a_i = t(\alpha_{i+1})$ for each $i \geq 0$. Then there exists some nonnegative integer $N$ such that $\abs{Q(a_i, a_j)} = 1$ and $\abs{[p]_{a_j}} = 1$ for any $i\geq j\geq N$.
\end{lemma}

\begin{proof}
  Since the convex hull of $p$ is uniformly interval finite, it contains no oriented cycles. In particular, $p$ is irrational.
  By the equivalence between Proposition~\ref{prop:uif}(1) and Proposition~\ref{prop:uif}(2), we have that $\set{\abs{Q(a_i, a_0)} \,\middle\vert i\geq0}$ is bounded.
  For any $i \geq j \geq l \geq0$, we have the injection
  \[
    f \colon Q(a_j, a_l) \times Q(a_i, a_j) \To Q(a_i, a_l), \quad
    (u, v) \mapsto u v.
  \]
  We obtain $\abs{Q(a_i, a_0)} \geq \abs{Q(a_j, a_0)} \times \abs{Q(a_i, a_j)} \geq \abs{Q(a_j, a_0)}$.
  Then there exists some nonnegative integer $N$ such that $\abs{Q(a_i, a_0)} = \abs{Q(a_N, a_0)}$ for any $i\geq N$.
  Since $\abs{Q(a_N, a_0)} \times \abs{Q(a_i, a_N)} \leq \abs{Q(a_i, a_0)}$, we have that $\abs{Q(a_i, a_N)} = 1$.
  Since $\abs{Q(a_j, a_N)} \times \abs{Q(a_i, a_j)} \leq \abs{Q(a_i, a_N)}$ for any $i \geq j \geq N$, we have that $\abs{Q(a_i, a_j)} = 1$.
  By Lemma~\ref{lem:[p]_a}, we obtain $\abs{[p]_{a_j}}=1$.
\end{proof}

We mention that the necessary conditions in the above lemma is not sufficient; see the following example.

\begin{example}
  \begin{enumerate}
    \item
      Let $Q$ be the following quiver
      \[\begin{tikzcd}[sep = 2.5em, /tikz/column 6/.style = {anchor = base west}]
        &\overset{b_1}{\circ}\ar[ld, bend right, "\beta_1"']
        &\overset{b_2}{\circ}\lar["\beta_2"']
        &\cdots              \lar["\beta_3"']
        &\overset{b_i}{\circ}\lar["\beta_i"']
        &\cdots              \lar["\beta_{i+1}"']
        \\
        \underset{a_0}{\circ}
        &\underset{a_1}{\circ}\lar["\alpha_1"] \uar["\gamma_1"']
        &\underset{a_2}{\circ}\lar["\alpha_2"] \uar["\gamma_2"']
        &\cdots               \lar["\alpha_3"]
        &\underset{a_i}{\circ}\lar["\alpha_i"] \uar["\gamma_i"']
        &\cdots.              \lar["\alpha_{i+1}"]
      \end{tikzcd}\]

      Let $\Omega$ be the set of equivalence classes of right infinite paths, and let $\Delta$ be the set of right infinite paths $p$ with $t(p) = a_0$.

      Set $u = \alpha_1 \alpha_2 \cdots \alpha_i \cdots$ and $v = \beta_1 \beta_2 \cdots \beta_i \cdots$.
      For each $i \geq1$, we set
      \[
        w_i =
        \beta_1 \beta_2 \cdots \beta_i \gamma_i
        \alpha_{i+1} \alpha_{i+2} \cdots \alpha_j \cdots.
      \]
      Then we have that
      \[
        \Delta = \set{u, v} \cup \set{w_i \middle\vert i \geq1}.
      \]

      We observe that each class in $\Omega$ contains some right infinite path in $\Delta$.
      Since $[u]_{a_0} = \set{u} \cup \set{w_i \middle\vert i \geq1}$ and $[v]_{a_0} = \set{v}$,
      we have that
      \[
        \Omega = \set{[u], [v]}.
      \]

      We observe that $\abs{Q(b_i, a_0)} = 1$ for any $i \geq0$.
      Then Proposition~\ref{prop:uif} implies that the convex hull of $v$ is uniformly interval finite.
      Moreover, the convex hull of any right infinite path in $[v]$ is uniformly interval finite.

      For each $i \geq1$, we have that $\abs{Q(a_i, a_0)} = i+1$.
      Then Proposition~\ref{prop:uif} implies that the convex hull of $u$ is not uniformly interval finite. Similarly, neither is the one of any $w_l$. But $\abs{Q(a_i, a_j)} = 1$ and $\abs{[p]_{a_j}} = 1$ for any $i\geq j\geq1$. Then the convex hull of each $\alpha_j \alpha_{j+1} \cdots \alpha_i \cdots$ is uniformly interval finite by Proposition~\ref{prop:uif}.
    \item
      Let $Q$ be the following quiver
      \[\begin{tikzcd}[sep = 2.5em]
        \overset{b_0}{\circ}\dar["\gamma_0"]
        &\overset{b_1}{\circ}\lar["\beta_1"']
        &\overset{b_2}{\circ}\lar["\beta_2"'] \dar["\gamma_2"]
        &\overset{b_3}{\circ}\lar["\beta_3"']
        &\overset{b_{2i}}{\circ}\lar[phantom, "\cdots"] \dar["\gamma_{2i}"]
        &\overset{b_{2i+1}}{\circ}\lar["\beta_{2i+1}"']
        &[-1.5em]\lar[phantom, "\cdots"]
        \\
        \underset{a_0}{\circ}
        &\underset{a_1}{\circ}\lar["\alpha_1"] \uar["\gamma_1"']
        &\underset{a_2}{\circ}\lar["\alpha_2"]
        &\underset{a_3}{\circ}\lar["\alpha_3"] \uar["\gamma_3"']
        &\underset{a_{2i}}{\circ}\lar[phantom, "\cdots"]
        &\underset{a_{2i+1}}{\circ}\lar["\alpha_{2i+1}"] \uar["\gamma_{2i+1}"']
        &.\lar[phantom, "\cdots"]
      \end{tikzcd}\]

      Let $\Delta$ be the set of right infinite paths $p$ with $t(p) = a_0$. Then $\Delta$ is uncountable. Let $\Omega$ be the set of equivalence classes of right infinite paths.
      We observe that each class in $\Omega$ contains at least one right infinite path in $\Delta$, and at most countably many right infinite paths in $\Delta$. Therefore $\Omega$ is uncountable.

      For any positive integer $N$, we have that $\abs{Q(a_{N+2}, a_N)} = 2$. Then Lemma~\ref{lem:uif} implies that the convex hull of $\alpha_1 \alpha_2 \cdots \alpha_i \cdots$ is not uniformly interval finite.

      Moreover, one can show that the convex hull of any right infinite path $p$ in $\Delta$ is not uniformly interval finite.
      Indeed, we observe that either $a_i$ or $b_i$ will appear in $p$ for any $i \geq 0$. For any positive integer $N$, we have that $\abs{Q(c, c')} \geq 2$ for any $c \in \set{a_{N+3}, b_{N+3}}$ and $c' \in \set{a_N, b_N}$. Then Lemma~\ref{lem:uif} implies that the convex hull of $p$ is not uniformly interval finite.
    \qed
  \end{enumerate}
\end{example}

\section{Representations of quivers}
\label{sec:rep}

Let $k$ be a field and let $Q$ be a quiver.
Denote by $\Mod k$ the category of $k$-linear spaces.

A representation $M = (M(a), M(\alpha))$ of $Q$ over $k$ is given by $k$-linear spaces $M(a)$ for any vertex $a$, and $k$-linear maps $M(\alpha) \colon M(a) \to M(b)$ for any arrow $\alpha \colon a \to b$.
For each finite path $p = \alpha_l \cdots \alpha_2 \alpha_1$ of length $l \geq 1$, we set
$M(p) = M(\alpha_l) \circ \cdots \circ M(\alpha_2) \circ M(\alpha_1)$;
for each trivial path $e_a$, we set $M(e_a) = \1_{M(a)}$.
Given two representations $M$ and $N$, a morphism $f \colon M \to N$ is given by $k$-linear maps $f(a) \colon M(a) \to N(a)$ for any vertex $a$, such that $f(b) \circ M(\alpha) = N(\alpha) \circ f(a)$ for any arrow $\alpha \colon a \to b$.
%\[\begin{tikzcd}
%  M(a) \rar["f(a)"] \dar["M(\alpha)"'] & N(a) \dar["N(\alpha)"]\\
%  M(b) \rar["f(b)"]                    & N(b)
%\end{tikzcd}\]

Let $M$ be a representation.
Recall that the \emph{support} $\supp M$ of $M$ is the full subquiver of $Q$ generated by vertices $a$ with $M(a)\neq0$.
The \emph{socle} $\soc M$ of $M$ is the subrepresentation of $M$ such that
\[
  (\soc M)(a) = \bigcap_{\alpha\in Q_1, s(\alpha)=a} \Ker M(\alpha)
\]
for each vertex $a$.
The \emph{radical} $\rad M$ of $M$ is the subrepresentation of $M$ such that
\[
  (\rad M)(a) = \sum_{\alpha\in Q_1, t(\alpha)=a} \Im M(\alpha)
\]
for each vertex $a$.
The \emph{top} $\top M$ of $M$ is the factor representation $M/\rad M$.

We mention the following observation;
compare \cite[Lemma~1.1]{BautistaLiuPaquette2013Representation}.

\begin{lemma}\label{lem:ess}
  Let $M$ be a representation.
  \begin{enumerate}
    \item \label{item:lem:ess:1}
      If $\supp M$ contains no right infinite paths, then $\rad M$ is superfluous in $M$.
    \item \label{item:lem:ess:2}
      If $\supp M$ contains no left infinite paths, then $\soc M$ is essential in $M$.
  \end{enumerate}
\end{lemma}

\begin{proof}
  (1) Let $N$ be a subrepresentation of $M$ with $N + \rad M = M$. We assume $N(a) \neq M(a)$ for some vertex $a$.

  We claim that there exists some $a_1 \in a^-$ such that $N(a_1) \neq M(a_1)$. Indeed, otherwise $(\rad M)(a) \subseteq N(a)$ and then $(\rad M + N)(a) = N(a) \neq M(a)$, which is a contradiction.

  Set $a_0 = a$. For every $i \geq 0$, we can find some $a_{i+1} \in a_i^-$ inductively such that $N(a_{i+1}) \neq M(a_{i+1})$. Then we obtain some right infinite path in $\supp M$, which is a contradiction.
  It follows that $N = M$, and then $\rad M$ is superfluous in $M$.

  (2) Let $N$ be a nonzero subrepresentation of $M$. Let $a$ be a vertex in $\supp N$. Assume $x$ is a nonzero element in $N(a)$. Since $\supp M$ contains no left infinite paths, there exists some finite path $p$ in $\supp M$ with $s(p)=a$ such that $N(p)(x) \neq 0$ and $N(\alpha p)(x) = 0$ for any arrow $\alpha$ in $Q$. Therefore $N(p)(x)$ lies in $N \cap \soc M$. It follows that $\soc M$ is essential in $M$.
\end{proof}

Denote by $\Rep(Q)$ the category of representations of $Q$ over $k$.
For every pair of representations $M$ and $N$, denote by $\Hom (M, N)$ the set of morphisms from $M$ to $N$ in $\Rep(Q)$.
It is well known that $\Rep(Q)$ is a hereditary abelian category;
see \cite[Section~8.2]{GabrielRoiter1992Representations}.

We associate each representation $M$ of $Q$ with a representation $D M$ of $Q^\op$ as follows.
For each vertex $a$ in $Q^\op$, we let
\[
  (D M) (a) = \Hom_k(M(a), k).
\]
For each arrow $\alpha^\op \colon b \to a$ in $Q^\op$, we let
\[
  (D M) (\alpha^\op) = \Hom_k(M(\alpha), k) \colon (D M) (b) \To (D M) (a).
\]
Given a morphism $f \colon M \to N$ in $\Rep(Q)$, we set the morphism $D f \colon D N \to D M$ in $\Rep(Q^\op)$ such that
\[
  (D f) (a) = \Hom_k(f(a), k) \colon (D N) (a) \To (D M) (a).
\]
Then we obtain an exact contravariant functor
\[
  D \colon \Rep(Q) \To \Rep(Q^\op).
\]

Let $a$ be a vertex in $Q$. We define the representation $P_a$ as follows.
For each vertex $b$, we let
\[
  P_a(b) = \bigoplus_{p \in Q(a,b)} k p.
\]
We mention that $P_a(b) = 0$ if $Q(a,b)=\emptyset$.
For each arrow $\alpha \colon b \to b'$, we let
\[
  P_a (\alpha) \colon P_a(b) \To P_a(b'), \quad
  p \mapsto \alpha p,
\]
for any finite path $p \in Q(a,b)$.

Denote by $P_a^\op$ the corresponding representation of $Q^\op$ and let $I_a = D P_a^\op$ in $\Rep(Q)$.
We mention that for each vertex $b$, we have that
\[
  I_a (b) = \Hom_k \biggl( \bigoplus_{p \in Q(b,a)} k p, k \biggr).
\]
For each arrow $\alpha \colon b \to b'$, we have that
\[
  I_a(\alpha) \colon I_a(b) \To I_a(b'), \quad
  f \mapsto (p \mapsto f(p \alpha)),
\]
for any $f \in I_a (b)$ and any finite path $p \in Q(b',a)$.

Let $S_a$ be the simple representation such that $S_a (a) = k e_a$ and $S_a (b) = 0$ for any vertex $b \neq a$. We observe that $\top P_a \simeq S_a \simeq \soc I_a$.

The following result seems well known;
see \cite[Section~3.7]{GabrielRoiter1992Representations}.
It implies that $P_a$ is a projective representation and $I_a$ is an injective representation in $\Rep(Q)$.

\begin{lemma}\label{lem:proj-inj}
  Let $M \in \Rep(Q)$ and $a \in Q_0$.
  \begin{enumerate}
    \item \label{item:lem:proj-inj:1}
      The $k$-linear map
      \[
        \eta_M \colon \Hom(P_a, M) \To M(a),
%        f \mapsto f (a) (e_a)
      \]
      given by $\eta_M(f) = f(a) (e_a)$ for any $f \in \Hom (P_a, M)$, is an isomorphism natural in $M$.
%      is an isomorphism and natural in $M$.
    \item \label{item:lem:proj-inj:2}
      The $k$-linear map
      \[
        \zeta_M \colon \Hom(M, I_a) \To \Hom_k (M(a), k),
%        f \mapsto (x \mapsto f(a) (x) (e_a))
      \]
      given by $\zeta_M(f)(x) = f(a) (x) (e_a)$ for any $f \in \Hom ( M, I_a )$ and $x \in M(a)$, is an isomorphism natural in $M$.
%      is an isomorphism and natural in $M$.
  \end{enumerate}
\end{lemma}

\begin{proof}
  (1) Consider the $k$-linear map
  \[
    \eta'_M \colon M(a) \To \Hom(P_a, M),
  \]
  given by $\eta'_M(x) (b) (p) = M(p)(x)$, for any $x \in M(a)$ and any $b \in Q_0$ and $p \in Q(a,b)$.
  By a direct verification, we have that
  \[
    \eta'_M \circ \eta_M = \1_{\Hom(P_a, M)}
    \quad \mbox{and} \quad
    \eta_M \circ \eta'_M = \1_{M(a)}.
  \]
  It follows that $\eta_M$ is an isomorphism. The naturality is a direct verification.

  (2) Consider the $k$-linear map
  \[
    \zeta'_M \colon \Hom_k( M(a), k ) \To \Hom ( M, I_a ),
  \]
  given by $\zeta'_M (f) (b) (x) (p) = f( M(p)(x) )$ for any $f \in \Hom_k( M(a), k )$ and any $b \in Q_0$, $x \in M(b)$ and $p \in Q(b,a)$.
  By a direct verification, we have that
  \[
    \zeta'_M \circ \zeta_M = \1_{\Hom( M, I_a )}
    \quad \mbox{and} \quad
    \zeta_M \circ \zeta'_M = \1_{\Hom_k( M(a), k )}.
  \]
  It follows that $\zeta_M$ is an isomorphism. The naturality is a direct verification.
\end{proof}

%\begin{proof}
%  (1) We have a $k$-linear map $\eta'_M \colon M(a) \to \Hom(P_a, M)$ such that
%  \[
%    \eta'_M (x) (b) \colon P_a(b) = \bigoplus_{p \in Q(a,b)} k p \To M(b),
%    p \mapsto M(p) (x),
%  \]
%  for any $x \in M(a)$ and any vertex $b$.
%  By a direct verification, we have that
%  \[
%    \eta'_M \circ \eta_M = \1_{\Hom(P_a, M)}
%    \mbox{ and }
%    \eta_M \circ \eta'_M = \1_{M(a)}.
%  \]
%  It follows that $\eta_M$ is an isomorphism. The naturality is a direct verification.
%
%  (2) We have a $k$-linear map $\zeta'_M \colon \Hom_k (M(a), k) \to \Hom(M, I_a)$ such that
%  \[
%    \zeta'_M (f) (b) (x) \colon \bigoplus_{p \in Q(b,a)} k p \To k,
%    p \mapsto ( f \circ M(p) ) (x),
%  \]
%  for any $k$-linear map $f \colon M(a) \to k$, any vertex $b$ and any $x \in M(b)$.
%  By a direct verification, we have that
%  \[
%    \zeta'_M \circ \zeta_M = \1_{\Hom(M, I_a)}
%    \mbox{ and }
%    \zeta_M \circ \zeta'_M = \1_{\Hom_k (M(a), k)}.
%  \]
%  It follows that $\zeta_M$ is an isomorphism. The naturality is a direct verification.
%\end{proof}

Let $M$ be a representation.
Then there exists some epimorphism $\bigoplus_{i \in \Lambda} P_{a_i} \to M$ with $a_i \in Q_0$.
We call $M$ \emph{countably generated} if $\Lambda$ can be chosen as a countable set; we call $M$ \emph{finitely generated} if $\Lambda$ can be chosen as a finite set.
Dually, there exists some monomorphism $M \to \prod_{i \in \Lambda} I_{a_i}$ with $a_i \in Q_0$.
We call $M$ \emph{countably cogenerated} if $\Lambda$ can be chosen as a countable set; we call $M$ \emph{finitely cogenerated} if $\Lambda$ can be chosen as a finite set.

Denote by $\proj(Q)$ the category of finitely generated projective representations and by $\inj(Q)$ the category of finitely cogenerated injective representations.

If $Q$ contains no oriented cycles, we observe by Lemma~\ref{lem:proj-inj} that $\End(P_a) \simeq k \simeq \End(I_a)$ for any vertex $a$.
Therefore, both $\proj(Q)$ and $\inj(Q)$ are Krull--Schmidt categories.
Every object in $\proj(Q)$ is of the form $\bigoplus_{i=1}^n P_{a_i}$ with $a_i \in Q_0$, and every object in $\inj(Q)$ is of the form $\bigoplus_{i=1}^n I_{a_i}$ with $a_i \in Q_0$.
Moreover, it follows from Azumaya's decomposition theorem that every projective representation in $\Rep(Q)$ is the direct sum of representations of the form $P_a$;
see \cite[Theorem~12.6]{AndersonFuller1974Rings}.

Recall that a representation $M$ is called \emph{pointwise finite dimensional} if $M(a)$ is finite dimensional for each vertex $a$.
We denote by $\rep(Q)$ the category of pointwise finite dimensional representations.
The restriction of $D$ gives a duality
\[
  D \colon \rep(Q) \To \rep(Q^\op).
\]

The following fact is well known.

\begin{lemma}\label{lem:here}
  $\rep(Q)$ is a hereditary abelian subcategory of $\Rep(Q)$, which is closed under extensions.
\end{lemma}

\begin{proof}
  One can see that $\rep(Q)$ is an abelian subcategory of $\Rep(Q)$, which is closed under extensions.
  For any $M$ and $N$ in $\rep(Q)$, we will view $\Ext^1(M,N)$ in the sense of Yoneda under the Baer sum.
  Since $\Rep(Q)$ is hereditary, the functors $\Ext^1(M,-)$ and $\Ext^1(-,M)$ from $\Rep(Q)$ to $\Mod k$ are right exact. Then so are their restrictions to $\rep(Q)$, since $\rep(Q)$ is closed under extensions. That is to say, $\rep(Q)$ is hereditary.
\end{proof}

Given infinitely many objects in $\rep(Q)$, it depends on their supports whether they admit a direct sum in $\rep(Q)$.

\begin{proposition}\label{prop:inf-sum}
  Let $M_i$ for $i \in \Lambda$ be infinitely many objects in $\rep(Q)$.
  The following statements are equivalent.
  \begin{enumerate}
    \item
      The product of $M_i$ for $i \in \Lambda$ exists in $\rep(Q)$.
    \item
      The coproduct of $M_i$ for $i \in \Lambda$ exists in $\rep(Q)$.
    \item
      For every vertex $a$, there exist only finitely many $i \in \Lambda$ such that $M_i (a) \neq 0$.
  \end{enumerate}
\end{proposition}

\begin{proof}
  (1) $\Rightarrow$ (3).
  Let $(M, f_i \colon M \to M_i)$ be the product.
  Assume there exist infinitely many $i \in \Lambda$ such that $M_i (a) \neq 0$, for some vertex $a$. From these elements, we can choose some $i_1, i_2, \dots, i_n$ such that $n > \dim M(a)$.

  Denote by $g_i \colon M_i \to M$ the induced morphism such that $f_i \circ g_i = \1_{M_i}$ and $f_j \circ g_i = 0$ for any $j \neq i$.
  Consider the morphisms $g \colon \bigoplus_{j=1}^n M_{i_j} \to M$ induced by $g_i$, and $f \colon M \to \bigoplus_{j=1}^n M_{i_j}$ induced by $f_i$.
  One can check that $f \circ g = \1_{\bigoplus_{j=1}^n M_{i_j}}$.
  In particular, $\dim M(a) \geq \dim \bigoplus_{j=1}^n M_{i_j} (a) \geq n$, which is a contradiction.

  (3) $\Rightarrow$ (1).
  By the assumption, $\prod_{i \in \Lambda} M_i(a)$ is a finite direct product for any vertex $a$. Then the direct product $\prod_{i \in \Lambda} M_i$ in $\Rep(Q)$ lies in $\rep(Q)$. It is the product of $M_i$ for $i \in \Lambda$ in $\rep(Q)$.

  The proof of (2) $\Leftrightarrow$ (3) is similar.
\end{proof}

If the equivalent conditions in the preceding proposition hold, then the product and coproduct are just the direct sum in $\Rep(Q)$, which we denote by $\bigoplus_{i \in \Lambda} M_i$.
Moreover, $\bigoplus_{i \in \Lambda} M_i(a)$ is a finite direct sum for every vertex $a$.

\begin{example}
  Let $Q$ be the following quiver
  \[\begin{tikzcd}[/tikz/column 6/.style = {anchor = base west}]
    \overset{1}{\circ} \dar
    &\overset{3}{\circ} \lar \dar
    &\overset{5}{\circ} \lar \dar
    &\cdots \lar
    &\overset{2n+1}{\circ} \lar \dar
    &\cdots \lar
    \\
    \underset{0}{\circ}
    &\underset{2}{\circ}
    &\underset{4}{\circ}
    &\cdots
    &\underset{2n}{\circ}
    &\cdots.
  \end{tikzcd}\]

  We observe that $P_{2n} \simeq S_{2n}$ for $n \geq 0$, and they admit a direct sum $\bigoplus_{n\geq0} P_{2n}$ in $\rep(Q)$.
  Since $P_{2n+1}(1) \neq 0$ for any $n \geq 0$, then $P_{2n+1}$ for $n \geq 0$ do not admit a direct sum in $\rep(Q)$ by Proposition~\ref{prop:inf-sum}. Indeed, $\bigoplus_{n\geq0} P_{2n+1}$ does not lie in $\rep(Q)$.
%  By Corollary~7.3, it is a flat representation. Moreover, one can construct it as the direct limit of a direct system of objects in $\proj(Q)$ over a countable directed set as follows.
%
%  Let $\Omega$ be the set of finite subsets of $2 \N$. Then $\Omega$ with the inclusions is a countable directed set.
%  For any $E \in \Omega$, we let $M_E = \bigoplus_{i \in E} P_i$, and let $\phi_E \colon M_E \to \bigoplus_{n\geq0} P_{2n}$ be the canonical injection. For any $E \subseteq F \subseteq 2 \N$, we let $\psi_{E F} \colon M_E \to M_F$ be the canonical injection.
%  Then $(M_E, \psi_{E F})$ is a direct system over $\Omega$ and $(\bigoplus_{n\geq0} P_{2n}, \phi_E)$ is the direct limit in $\Rep(Q)$.
  \qed
\end{example}

The following fact may be known to experts.
One can see \cite[Theorem~1]{JiaoDecompositions} for details;
compare \cite[Main Theorem(a)]{AngeleriHugeldelaPena2009Locally}.

\begin{lemma}\label{lem:dec}
  Every $M \in \rep(Q)$ admits a decomposition $M = \bigoplus_{i \in \Lambda} M_i$ such that each $M_i$ is indecomposable and $\End(M_i)$ is local.
  \qed
\end{lemma}

\section{Direct limits and inverse limits}
\label{sec:limit}

In this section, we mention some facts about direct limits and inverse limits.

Let $(M_i, \psi_{ij}\colon M_i \to M_j)$ be a direct system over a directed set $(\Lambda, \leq)$ in $\Mod k$.
For each $i \in \Lambda$, we set
\[
  M'_i = \sum_{i \leq j} \Ker \psi_{ij}.
\]

We mention the following observations.

\begin{lemma}\label{lem:M'i}
  For any $i \in \Lambda$ and $x \in M'_i$, there exists some $l \geq i$ such that $x \in \Ker \psi_{il}$.
\end{lemma}

\begin{proof}
  Assume $x = \sum_{r=1}^n x_r$ with each $x_r \in \Ker \psi_{i j_r}$ for some $j_r \geq i$. We can choose some $l \geq j_1, j_2, \dots, j_n$, since $\Lambda$ is directed.
  We observe that
  \[
    \psi_{il} (x)
    = \sum_{r=1}^n \psi_{il} (x_r)
    = \sum_{r=1}^n (\psi_{j_r l} \circ \psi_{i j_r}) (x_r)
    = 0.
  \]
  Then the result follows.
\end{proof}

\begin{lemma}\label{lem:colim-new}
  For any $i \leq j$, the pre-image of $M'_j$ under $\psi_{ij}$ is $M'_i$.
\end{lemma}

\begin{proof}
  Let $x \in M'_i$. Applying Lemma~\ref{lem:M'i} for $i$ and $x$, we have some $l \geq i$ such that $x \in \Ker \psi_{il}$. Choose some $l' \geq j, l$. Then
  \[
    (\psi_{jl'} \circ \psi_{ij}) (x)
    = \psi_{il'} (x)
    = (\psi_{ll'} \circ \psi_{il}) (x)
    = 0.
  \]
  Hence $\psi_{ij} (x) \in \Ker \psi_{jl'} \subseteq M'_j$.
  That is to say, the image of $M'_i$ under $\psi_{ij}$ is contained in $M'_j$.

  On the other hand, let $x \in M_i$ with $\psi_{ij}(x) \in M'_j$. Applying Lemma~\ref{lem:M'i} for $j$ and $\psi_{ij}(x)$, we have some $l \geq j$ such that $\psi_{ij}(x) \in \Ker \psi_{jl}$. Then
  \[
    \psi_{il} (x) = (\psi_{jl} \circ \psi_{ij}) (x) = 0.
  \]
  Hence $x \in \Ker \psi_{il} \subseteq M'_i$.
  Then the result follows.
\end{proof}

By Lemma~\ref{lem:colim-new}, for any $i \leq j$, we can consider the restriction $\psi_{ij}|_{M'_i} \colon M'_i \to M'_j$ of $\psi_{ij}$, and the morphism $\overline{\psi_{ij}} \colon M_i/M'_i \to M_j/M'_j$ induced by $\psi_{ij}$.
Then we obtain direct systems $(M'_i, \psi_{ij}|_{M'_i})$ and $(M_i/M'_i, \overline{\psi_{ij}})$ over $\Lambda$.
Moreover, it follows from Lemma~\ref{lem:colim-new} that each $\overline{\psi_{ij}}$ is an injection.

The following result shows that the direct system $(M_i/M'_i, \overline{\psi_{ij}})$ admits the same direct limit with $(M_i, \psi_{ij})$.
%We then obtain a new direct system holding the direct limits, such that its morphisms are both injections.

\begin{proposition}\label{prop:colim-mono}
  The canonical map $\phi_i \colon M_i/M'_i \to \varinjlim M_i/M'_i$ is an injection for each $i \in \Lambda$, and $\varinjlim M_i/M'_i \simeq \varinjlim M_i$.
\end{proposition}

\begin{proof}
  For any $i \in \Lambda$ and any $x + M'_i \in M_i/M'_i$, we mention that $\phi_i (x + M'_i) = 0$ if and only if $\overline{\psi_{il}} (x + M'_i) = 0$ for some $l \geq i$;
  see \cite[Lemma~5.30]{Rotman2009introduction}.
  We observe by Lemma~\ref{lem:colim-new} that $\overline{\psi_{ij}}$ is an injection for any $j \geq i$.
  It follows that $\phi_i$ is an injection.

  Consider the following exact sequence of direct systems
  \[
    0 \To (M'_i) \To (M_i) \To (M_i/M'_i) \To 0.
  \]
  We obtain by \cite[Theorem~2.6.15]{Weibel1994introduction} the exact sequence
  \[
    0 \To \varinjlim M'_i \To \varinjlim M_i \To \varinjlim M_i/M'_i \To 0.
  \]
  By Lemma~\ref{lem:M'i}, for any $i \in \Lambda$ and $x \in M'_i$, there exists some $j \geq i$ such that $x \in \Ker \psi_{ij}$. That is to say $\psi_{ij} (x) = 0$. We then obtain $\varinjlim M'_i = 0$. It follows that $\varinjlim M_i/M'_i \simeq \varinjlim M_i$.
\end{proof}

Based on the above procedure, we can reduce the structure morphisms in a direct system into monomorphisms;
see Proposition~\ref{prop:flat}.

Let $(M_i, \psi_{ji}\colon M_j \to M_i)$ be an inverse system over $\Lambda$ in $\Mod k$.
It admits the inverse limit $\varprojlim M_i$, which is the subspace of $\prod_{i \in \Lambda} M_i$ consisting of threads;
see \cite[Proposition~5.17]{Rotman2009introduction}.
Here, a \emph{thread} means an element $(m_i) \in \prod_{i \in \Lambda} M_i$ such that $\psi_{ji}(m_j) = m_i$ for any $i \leq j$.

Recall that $(M_i, \psi_{ji})$ is said to satisfy the \emph{Mittag-Leffler condition}, if for each $i \in \Lambda$, there exists some $j \geq i$ such that $\psi_{ji} (M_j) = \psi_{li} (M_l)$ for any $l \geq j$.
We mention that if each $M_i$ is finite dimensional, then $(M_i, \psi_{ji})$ satisfies the Mittag-Leffler condition naturally.

The following fact is well known;
see \cite[Proposition~3.5.7]{Weibel1994introduction}.

\begin{lemma}\label{lem:lim-exact}
  Assume
  \[
    0 \To (U_i) \To (V_i) \To (W_i) \To 0
  \]
  is a short exact sequence of inverse systems over $\Lambda$ in $\Mod k$. Then the following induced sequence of inverse limits
  \[
    0 \To \varprojlim U_i \To \varprojlim V_i \overset{g}{\To} \varprojlim W_i
  \]
  is exact. If moreover $(U_i)$ satisfies the Mittag-Leffler condition and $\Lambda$ is countable, then $g$ is surjective.
  \qed
\end{lemma}

The following result is crucial. It will be used to show the projective property of some flat representation in $\rep(Q)$; see Section~\ref{sec:proj}.

\begin{proposition}\label{prop:ft}
  Let $\mathcal{C}$ be a $k$-linear category and let $\Lambda$ be a countable directed set. Assume that $(M_i, \psi_{ij}\colon M_i \to M_j)$ is a direct system over $\Lambda$ in $\mathcal{C}$, with the direct limit $(M,\phi_i\colon M_i\to M)$. If some morphism $h\colon M\to V$ satisfies that each morphism $h \circ \phi_i$ factors through $g\colon U\to V$ and each $k$-linear space $\Homc(M_i, U)$ is finite dimensional, then $h$ factors through $g$.
\end{proposition}

\[\begin{tikzcd}[column sep = 3em]
  &&M \ar[dd,"h"]% \ar[ddl, out=245, in=40, dashed]
  \\
  M_i \rar["\psi_{ij}"]
    \ar[rru, out= 50, in=180, "\phi_i" near start]
    \ar[rd,  out=-50, in=165, dashed]
  &M_j
    \ar[ru, "\phi_j"]
    \ar[d, dashed]
  \\
  &U \rar["g"]  &V
\end{tikzcd}\]

\begin{proof}
  For any $i \leq j$ and $W \in \mathcal{C}$, we consider the $k$-linear map
  \[
    \Homc(\psi_{ij}, W) \colon \Homc(M_j, W) \To \Homc(M_i, W), \quad
    f \mapsto f \circ \psi_{ij},
  \]
  induced by $\psi_{ij}$.
  Then $(\Homc(M_i, W), \Homc(\psi_{ij}, W))$ forms an inverse system in $\Mod k$. We mention the natural isomorphism
  \begin{equation}\label{eq:limit}
    \theta_W \colon
    \Homc(M, W) \overset{\simeq}{\To} \varprojlim \Homc(M_i, W), \quad
    f \mapsto ( f \circ \phi_i ).
  \end{equation}

  For each $i \in \Lambda$, we have the $k$-linear map
  \[
    \Homc(M_i, g) \colon \Homc(M_i, U) \To \Homc(M_i, V), \quad
    f \mapsto g \circ f,
  \]
  induced by $g$.
  Express it as the composition of a surjection and the inclusion
  \[
    \Homc(M_i, g) \colon \Homc(M_i, U) \overset{u_i}{\To}
    \Im \Homc(M_i, g) \overset{\inc_i}{\subseteq} \Homc(M_i, V).
  \]
  Consider the induced exact sequence
  \[
    0 \To \Ker \Homc(M_i, g) \To \Homc(M_i, U) \overset{u_i}{\To} \Im \Homc(M_i, g) \To 0.
  \]
  We observe that each $\Ker \Homc(M_i, g)$ is finite dimensional. Then the inverse system $(\Ker \Homc(M_i, g))$ satisfies the Mittag-Leffler condition. By Lemma~\ref{lem:lim-exact}, we obtain the following exact sequence of inverse limits
  \[
    0 \to \varprojlim \Ker \Homc(M_i, g) \to
    \varprojlim \Homc(M_i, U) \xrightarrow{\overleftarrow{u_i}}
    \varprojlim \Im \Homc(M_i, g) \to 0.
  \]
  Here, $\overleftarrow{u_i}$ is the induced map.
  Then we obtain the commutative triangle
  \[\begin{tikzcd}[column sep = -1.5em]
    &\varprojlim \Im \Homc(M_i, g)\\
    \varprojlim \Homc(M_i, U) &&\varprojlim \Homc(M_i, V).
    \ar[from=2-1, to=2-3, "{\overleftarrow{\Homc(M_i, g)}}"]
    \ar[from=2-1, to=1-2, start anchor= 30, two heads, "\overleftarrow{u_i}"]
    \ar[from=1-2, to=2-3, end   anchor=150, hook, "\overleftarrow{\inc_i}"]
  \end{tikzcd}\]
  Here, $\overleftarrow{\inc_i}$ is an injection by Lemma~\ref{lem:lim-exact}.

  For any $i \leq j$, we observe that
  \[
    \Homc(\psi_{ij}, V) (h \circ \phi_j)
    = h \circ \phi_j \circ \psi_{ij}
    = h \circ \phi_i.
  \]
  Then $( h \circ \phi_i )$ is an element in $\varprojlim \Homc(M_i, V)$.
  By the hypotheses, each $h \circ \phi_i$ factors through $g$. In other words, $h \circ \phi_i$ lies in $\Im \Homc(M_i, g)$.
  Then $( h \circ \phi_i )$ is also an element in $\varprojlim \Im \Homc(M_i, g)$.
  Since $\overleftarrow{u_i}$ is a surjection, there exists some $x \in \varprojlim \Homc(M_i, U)$ such that $\overleftarrow{u_i} (x) = ( h \circ \phi_i )$. Moreover
  \[
    \overleftarrow{\Homc(M_i, g)} (x)
    = \overleftarrow{\inc_i} ( \overleftarrow{u_i} (x) )
    = ( \inc_i( h \circ \phi_i ) ) = ( h \circ \phi_i ).
  \]
  In other words, $( h \circ \phi_i )$ lies in $\Im \overleftarrow{\Homc(M_i, g)}$.

  Consider the commutative diagram
  \[\begin{tikzcd}[sep = large, column sep = 5.5em]
    \Homc(M, U)\dar["\theta_U"'] &\Homc(M, V)\dar["\theta_V"]\\
    \varprojlim \Homc(M_i, U)    &\varprojlim \Homc(M_i, V).
    \ar[from=1-1, to=1-2, "{\Homc(M, g)}"]
    \ar[from=2-1, to=2-2, "{\overleftarrow{\Homc(M_i, g)}}"]
  \end{tikzcd}\]
  We observe by the isomorphism (\ref{eq:limit}) in the case $\theta_V$ that $\theta_V (h) = ( h \circ \phi_i )$. It follows that $h$ lies in $\Im \Homc(M, g)$. In other words, $h$ factors through $g$.
\end{proof}

\section{A class of flat representations}
\label{sec:Xp-Yp}

Let $Q$ be a quiver without oriented cycles. We mention that both $\proj(Q)$ and $\inj(Q)$ are Krull--Schmidt categories.

Recall the Lazard--Govorov Theorem that a module over a unital ring is flat if and only if it is a direct limit of finitely generated free modules; see \cite{Govorov1965flat,Lazard1964Sur}.
%compare \cite[Theorem~I.1.2]{Lazard1969Autour}.
It is generalized to functor categories as \cite[Theorem~3.2]{OberstRohrl1970Flat}.
Especially, a flat representation in $\Rep(Q)$ is a direct limit of finitely generated projective representations.
We will introduce a class of flat representations.

Let $[p]$ be an equivalence class of right infinite paths. %with some representative $p$.
Given a vertex $a$, we recall that $[p]_a$ is the subset of $[p]$ consisting of right infinite paths $q$ with $t(q) = a$.

Inspired by the irreducible representations of Leavitt path algebras introduced in
\cite[Subsection~3.1]{Chen2015Irreducible}
(compare the point modules studied in \cite{Smith2014space}),
we define a representation $X_{[p]}$ as follows.
For each vertex $a$, we let
\[
  X_{[p]}(a) = \bigoplus_{u \in[p]_a} k u.
\]
We mention that $X_{[p]} (a) = 0$ if $[p]_a = \emptyset$.
For each arrow $\alpha \colon a \to b$, we let
\[
  X_{[p]} (\alpha) \colon X_{[p]} (a) \To X_{[p]} (b), \quad
  u \mapsto \alpha u,
\]
for any right infinite path $u \in [p]_a$.

Assume $p = \alpha_1 \alpha_2 \cdots \alpha_i \cdots$. We set $a_i = t(\alpha_{i+1})$ for each $i\geq0$.
Recall that $P_a$ is the corresponding projective representation for any vertex $a$.
For each $j > i$, the finite path $\alpha_{i+1} \alpha_{i+2} \cdots \alpha_j$ induces a morphism $\psi_{ij} \colon P_{a_i} \to P_{a_j}$, such that
\[
  \psi_{ij} (b) \colon P_{a_i} (b) \To P_{a_j} (b), \quad
  u \mapsto u \alpha_{i+1} \alpha_{i+2} \cdots \alpha_j,
\]
for any vertex $b$ and any finite path $u \in Q(a_i, b)$.
Set $\psi_{ii} = \1_{P_{a_i}}$. We observe that $(P_{a_i}, \psi_{ij})$ is a direct system over $(\N, \leq)$.

For each $i \geq 0$, the right infinite path $\alpha_{i+1} \alpha_{i+2} \cdots \alpha_j \cdots$ induces a morphism $\phi_i \colon P_{a_i} \to X_{[p]}$, such that
\[
  \phi_i (b) \colon P_{a_i} (b) \To X_{[p]} (b), \quad
  u \mapsto u \alpha_{i+1} \alpha_{i+2} \cdots \alpha_j \cdots,
\]
for any vertex $b$ and any finite path $u \in Q(a_i, b)$.

The following result shows that $X_{[p]}$ is a flat representation in $\Rep(Q)$.

\begin{lemma}\label{lem:Xp-flat}
  $(X_{[p]}, \phi_i)$ is the direct limit of $(P_{a_i}, \psi_{ij})$ in $\Rep(Q)$.
\end{lemma}

\begin{proof}
We observe that $\phi_j \circ \psi_{ij} = \phi_i$ for any $i \leq j$.
Since $Q$ contains no oriented cycles, the right infinite path $\alpha_{i+1} \alpha_{i+2} \cdots \alpha_j \cdots$ is not cyclic. Lemma~\ref{lem:cycle} implies that each $\phi_i (b)$ is an injection. Then $\Im \phi_i \simeq P_{a_i}$.
Since $\bigcup_{i \geq 0} \Im \phi_i = X_{[p]}$, then the result follows.
\end{proof}

We observe that the direct system $(P_{a_i}, \psi_{ij})$ depends on the choice of representative $p$, but the direct limit $X_{[p]}$ does not.

We mention that $X_{[p]} \not\simeq P_a$ for any vertex $a$.
Indeed, we observe that $\rad X_{[p]} = X_{[p]}$ and hence $\top X_{[p]} = 0$. Then the result follows since $\top P_a \simeq S_a$.

%Then Lemma~\ref{lem:Xp-flat} implies that $X_{[p]}$ is flat. Moreover, it is the direct limit of finitely generated projective subrepresentations.
One can see from Lemma~\ref{lem:Xp-flat} that the flat representation $X_{[p]}$ is the direct limit of finitely generated projective subrepresentations.
The following result shows that this property holds for general flat representations in $\Rep(Q)$.
It strengthens the Lazard--Govorov Theorem in the special case $\Rep(Q)$.

\begin{proposition}\label{prop:flat}
  A flat representation in $\Rep(Q)$ is a direct limit of finitely generated projective subrepresentations.
\end{proposition}

\begin{proof}
  Assume $M$ is the direct limit of a direct system $(M_i, \psi_{ij} \colon M_i \to M_j)$ over a directed set $\Lambda$ in $\Rep(Q)$ with each $M_i \in \proj(Q)$. We set $M'_i = \sum_{i \leq j} \Ker \psi_{ij}$ for each $i \in \Lambda$.
  By Lemma~\ref{lem:colim-new}, we can let $\overline{\psi_{ij}} \colon M_i/M'_i \to M_j/M'_j$ be the morphism induced by $\psi_{ij}$, and then $(M_i/M'_i, \overline{\psi_{ij}})$ forms a direct system.
  By Proposition~\ref{prop:colim-mono}, we have that $M \simeq \varinjlim M_i/M'_i$ and each canonical morphism $M_i/M'_i \to M$ is an injection.
  It remains to show that $M_i/M'_i \in \proj(Q)$ for any $i \in \Lambda$.

  Let $i \in \Lambda$. We observe that $\Im \psi_{ij}$ is projective for any $j \geq i$, since $M_j$ is projective and $\Rep(Q)$ is hereditary.
  We obtain a split exact sequence
  \begin{equation}\label{eq:exact}
    0 \To \Ker \psi_{ij} \overset{u_j}{\To} M_i \To \Im \psi_{ij} \To 0.
  \end{equation}
  Here $u_j$ is the inclusion.
  Then $\Im \psi_{ij}$ and $\Ker \psi_{ij}$ are direct summands of $M_i$, and hence lie in $\proj(Q)$.
  We mention that there exists some morphism $v_j \colon M_i \to \Ker \psi_{ij}$ such that $v_j \circ u_j = \1_{\Ker \psi_{ij}}$.

  For any $j' \geq j \geq i$, we can consider the inclusions $u_{j'} \colon \Ker \psi_{ij'} \to M_i$ and $w \colon \Ker \psi_{ij} \to \Ker \psi_{ij'}$. We observe that $u_{j'} \circ w = u_j$ and hence
  $(v_j \circ u_{j'}) \circ w = v_j \circ u_j = \1_{\Ker \psi_{ij}}$.
  It follows that $\Ker \psi_{ij}$ is a direct summand of $\Ker \psi_{ij'}$.
  We may assume $M_i \simeq \bigoplus_{r=1}^n P_{a_r}$ for some vertices $a_1, a_2, \dots, a_n$. Then the number of indecomposable direct summands of $\Ker \psi_{ij}$ for a certain decomposition is less than or equal to the one of $\Ker \psi_{ij'}$, and they are both less than or equal to $n$. It follows that there exists some $l \geq i$ such that $\Ker \psi_{il} = \Ker \psi_{ij}$ for any $j \geq l$.

  For any $j \geq i$, we can choose some $j' \geq j, l$. We have that
  $\Ker \psi_{ij} \subseteq \Ker \psi_{ij'} = \Ker \psi_{il}$.
  We then obtain $M'_i = \Ker \psi_{il}$.
  It follows from the exact sequence (\ref{eq:exact}) in the case $\psi_{il}$ that
  $M_i/M'_i = M_i/\Ker \psi_{il} \simeq \Im \psi_{il}$
  and hence $M_i/M'_i \in \proj(Q)$.
\end{proof}

We study morphisms between representations of the form $X_{[p]}$.

\begin{lemma}\label{lem:pairwise}
  Let $[p]$ and $[q]$ be equivalence classes of right infinite paths. Assume $g \colon X_{[p]} \to X_{[q]}$ is a nonzero morphism. Then $[p] = [q]$ and $g = r \1_{X_{[p]}}$ for some nonzero $r \in k$.
\end{lemma}

\begin{proof}
  Assume $p=\alpha_1 \alpha_2 \cdots \alpha_i \cdots$. For each $i \geq 0$, we set $a_i = t(\alpha_{i+1})$.
  We observe that for any vertex $a$ and any right infinite path $p'$ in $[p]_{a}$, there exists some $i \geq 0$ and some finite path $w \in Q(a_i, a)$ such that $p' = w \alpha_{i+1} \alpha_{i+2} \cdots \alpha_l \cdots$.
  Consider the commutative diagram
  \[\begin{tikzcd}[sep = large]
    X_{[p]} (a_i) &X_{[q]} (a_i)\\
    X_{[p]} (a)   &X_{[q]} (a).
    \ar[from=1-1, to=2-1, "X_{[p]} (w)"']
    \ar[from=1-2, to=2-2, "X_{[q]} (w)"]
    \ar[from=1-1, to=1-2, "g(a_i)"]
    \ar[from=2-1, to=2-2, "g(a)"]
  \end{tikzcd}\]
  We observe that
  $p' = X_{[p]} (w) (\alpha_{i+1} \alpha_{i+2} \cdots \alpha_l \cdots)$.
  It follows that
  \begin{equation}\label{eq:g(a)}
    \begin{split}
      g (a) (p')
      &= ( g (a) \circ X_{[p]} (w) )
         (\alpha_{i+1} \alpha_{i+2} \cdots \alpha_l \cdots)\\
      &= ( X_{[q]} (w) \circ g (a_i) )
         (\alpha_{i+1} \alpha_{i+2} \cdots \alpha_l \cdots).
    \end{split}
  \end{equation}

  We claim that there exists some $i \geq 0$ such that
  $g (a_i) (\alpha_{i+1} \alpha_{i+2} \cdots \alpha_l \cdots) \neq 0$.
  Otherwise, for any vertex $a$ and any right infinite path $p'$ in $[p]_{a}$, we have $g (a) (p') = 0$ by (\ref{eq:g(a)}). It follows that $g = 0$, which is a contradiction.

  Hence, without loss of generality, we may assume $g(a_0) (p) \neq 0$.
  Then for any $i>0$, we observe by (\ref{eq:g(a)}) in the case $a = a_0$, $p' = p$ and $w = \alpha_1 \alpha_2 \cdots \alpha_i$ that
  \[
    0 \neq g (a_0) (p) =
    ( X_{[q]} (\alpha_1 \alpha_2 \cdots \alpha_i) \circ g (a_i) )
    (\alpha_{i+1} \alpha_{i+2} \cdots \alpha_l \cdots).
  \]
  In particular, we have $g(a_i) (\alpha_{i+1} \alpha_{i+2} \cdots \alpha_l \cdots) \neq 0$.

  For each $i \geq 0$, we assume
  \begin{equation}\label{eq:g(a_i)}
    g (a_i) (\alpha_{i+1} \alpha_{i+2} \cdots \alpha_l \cdots)
    = \sum_{j=1}^{n_i} r_{ij} u_{ij},
  \end{equation}
  where $n_i \geq 1$, $0 \neq r_{ij} \in k$, and $u_{ij} \in [q]_{a_i}$.

  We observe that
  \[\begin{split}
    \sum_{j=1}^{n_i} r_{ij} u_{ij}
    &= g (a_i) (\alpha_{i+1} \alpha_{i+2} \cdots \alpha_l \cdots)\\
    &= ( X_{[q]} (\alpha_{i+1}) \circ g (a_{i+1}) )
       (\alpha_{i+2} \alpha_{i+3} \cdots \alpha_l \cdots)\\
    &= X_{[q]} (\alpha_{i+1})
       \biggl( \sum_{j=1}^{n_{i+1}} r_{i+1,j} u_{i+1,j} \biggr)\\
    &= \sum_{j=1}^{n_{i+1}} r_{i+1,j} \alpha_{i+1} u_{i+1,j}.
  \end{split}\]
  Here, the first equality is (\ref{eq:g(a_i)}).
  The second equality is (\ref{eq:g(a)}) in the case $a = a_i$, $p' = \alpha_{i+1} \alpha_{i+2} \cdots \alpha_l \cdots$ and $w = \alpha_{i+1}$.
  The third equality holds by (\ref{eq:g(a_i)}) in the case $i+1$.
  The fourth equality holds by the definition of $X_{[q]} (\alpha_{i+1})$.

  Therefore, we have that $n_i = n_{i+1}$. For each $1 \leq j \leq n_i$, there exists some $1 \leq j' \leq n_{i+1}$ such that $u_{ij} = \alpha_{i+1} u_{i+1, j'}$ and $r_{ij} = r_{i+1, j'}$.

  By induction, we have that
  \[
    u_{0j} = \alpha_1 \alpha_2 \cdots \alpha_l \cdots = p
  \]
  for each $1 \leq j \leq n_0$ and hence $n_0 = 1$.
  It follows that $[p] = [q]$ since $u_{0j} \in [q]$.

  Let $r = r_{01}$. Then for each $i \geq 0$, we have
  \begin{equation}\label{eq:g(a_i)'}
    g (a_i) (\alpha_{i+1} \alpha_{i+2} \cdots \alpha_l \cdots)
    = r \alpha_{i+1} \alpha_{i+2} \cdots \alpha_l \cdots.
  \end{equation}

  For any vertex $a$ and any right infinite path $p'$ in $[p]_{a}$, we may assume $p' = w \alpha_{i+1} \alpha_{i+2} \cdots \alpha_l \cdots$ for some $i \geq 0$ and some finite path $w \in Q(a_i, a)$.
  By (\ref{eq:g(a)}) and (\ref{eq:g(a_i)'}), we have that
  \[\begin{split}
    g (a) (p')
    &= ( X_{[p]} (w) \circ g (a_i) )
       (\alpha_{i+1} \alpha_{i+2} \cdots \alpha_l \cdots)\\
    &= X_{[p]} (w)
       (r \alpha_{i+1} \alpha_{i+2} \cdots \alpha_l \cdots)\\
    &= r p'.
  \end{split}\]
  It follows that $g = r \1_{X_{[p]}}$.
\end{proof}

We mention that Lemma~\ref{lem:pairwise} still holds even if $Q$ contains oriented cycles.
The following result is a direct consequence of Lemma~\ref{lem:pairwise}.
It implies that the representations of the form $X_{[p]}$ are indecomposable and pairwise non-isomorphic.

\begin{proposition}\label{prop:indec}
  Let $[p]$ and $[q]$ be equivalence classes of right infinite paths.
  \begin{enumerate}
    \item \label{item:prop:indec:1}
      $X_{[p]} \simeq X_{[q]}$ if and only if $[p] = [q]$.
    \item \label{item:prop:indec:2}
      $\End(X_{[p]}) \simeq k$.
    \qed
  \end{enumerate}
\end{proposition}

We mention the following characterization for $X_{[p]}$ lying in $\rep(Q)$.

\begin{lemma}\label{lem:Xp-pf}
  Let $[p]$ be an equivalence class of right infinite paths. Then $X_{[p]}$ lies in $\rep(Q)$ if and only if the convex hull of any right infinite path in $[p]$ is uniformly interval finite.
\end{lemma}

\begin{proof}
  We observe that $X_{[p]}$ lies in $\rep(Q)$ if and only if $[p]_a$ is finite for any $a \in Q_0$. Since $Q$ contains no oriented cycles, every right infinite path is irrational. Then the result follows from the equivalence between Proposition~\ref{prop:uif}(1) and Proposition~\ref{prop:uif}(3).
\end{proof}

We mention that the convex hull of any right infinite path in $[p]$ is a convex subquiver of $\supp X_{[p]}$. But $X_{[p]} \in \rep(Q)$ does not imply that $\supp X_{[p]}$ is uniformly interval finite; see the following example.

\begin{example}
  Let $Q$ be the following quiver
  \[\begin{tikzcd}
    \cdots
    &\underset{-2}{\circ} \lar[yshift=.3em] \lar[yshift=-.3em]
    &\underset{-1}{\circ} \lar[yshift=.3em] \lar[yshift=-.3em]
    &\underset{ 0}{\circ} \lar[yshift=.3em] \lar[yshift=-.3em]
    &\underset{ 1}{\circ} \lar["\alpha_1"']
    &\underset{ 2}{\circ} \lar["\alpha_2"']
    &\cdots.              \lar["\alpha_3"']
  \end{tikzcd}\]
  Let $p = \alpha_1 \alpha_2 \cdots \alpha_i \cdots$. Then $\supp X_{[p]} = Q$. We observe that $\abs{[p]_i} = 1$ for each $i \geq 0$ and $\abs{Q(0,i)} = \abs{[p]_i} = 2^{-i}$ for each $i < 0$. It follows that $X_{[p]} \in \rep(Q)$. But $\supp X_{[p]}$ is not uniformly interval finite.
  \qed
\end{example}

Let $[p]$ be an equivalence class of left infinite paths.
Then $[p^\op]$ is the equivalence class of right infinite paths in $Q^\op$.
Assume $p = \cdots \alpha_i \cdots \alpha_2 \alpha_1$. Set $a_i = s(\alpha_{i+1})$ for each $i \geq 0$. Then $p^\op = \alpha_1^\op \alpha_2^\op \cdots \alpha_i^\op \cdots$ and $a_i = t(\alpha_{i+1}^\op)$ for each $i \geq 0$.
We have a direct system $(P_{a_i}^\op)$ over $\N$ with the direct limit $X_{[p^\op]}$ in $\Rep(Q^\op)$. Applying the contravariant functor $D \colon \Rep(Q^\op) \to \Rep(Q)$, we obtain an inverse system $(I_{a_i})$ over $\N$ in $\Rep(Q)$.
Let $Y_{[p]} = D X_{[p^\op]}$ in $\Rep(Q)$.

We observe that $D$ sends direct limits to inverse limits.
Then we obtain the following result, which is dual to Lemma~\ref{lem:Xp-flat}.

\begin{lemma}\label{lem:Yp-flat'}
  $Y_{[p]}$ is the inverse limit of $(I_{a_i})$ in $\Rep(Q)$.
  \qed
\end{lemma}

We mention that $Y_{[p]} \not\simeq I_a$ for any vertex $a$.
Indeed, we observe that $\soc Y_{[p]} = 0$. Then the result follows since $\soc I_a \simeq S_a$.

\section{The classification theorem}
\label{sec:proj}

Let $Q$ be an interval finite quiver. Then $\proj(Q)$ and $\inj(Q)$ are both subcategories of $\rep(Q)$.
We refer to \cite{BautistaLiuPaquette2013Representation} for more details about $\rep(Q)$.

\begin{lemma}\label{lem:flat-proj}
  Let $M \in \rep(Q)$ and let $\Lambda$ be a countable directed set. Assume that $M$ is a direct limit of objects $M_i \in \proj(Q)$ over $\Lambda$ in $\rep(Q)$. Then $M$ is projective in $\rep(Q)$.
\end{lemma}

\begin{proof}
  Assume that $\phi_i \colon M_i \to M$ is the canonical morphism for each $i \in \Lambda$.
  Let $g \colon X \to Y$ be an epimorphism and let $f \colon M \to Y$ be a morphism in $\rep(Q)$. We have that each $f \circ \phi_i$ factors through $g$ since $M_i$ is projective. We observe that $\Hom(M_i, X)$ is finite dimensional. By Proposition~\ref{prop:ft}, we have that $f$ factors through $g$. It follows that $M$ is projective in $\rep(Q)$.
\end{proof}

Then we can show that $X_{[p]}$ is projective in $\rep(Q)$ if it lies in $\rep(Q)$.

\begin{proposition}\label{prop:Xp-proj}
  Let $[p]$ be an equivalence class of right infinite paths such that the convex hull of any right infinite path in $[p]$ is uniformly interval finite. Then $X_{[p]}$ is an indecomposable projective object in $\rep(Q)$.
\end{proposition}

\begin{proof}
  By Lemma~\ref{lem:Xp-pf} and Propositions~\ref{prop:indec}(2), we have that $X_{[p]}$ is an indecomposable object in $\rep(Q)$. By Lemma~\ref{lem:Xp-flat}, $X_{[p]}$ is a direct limit of objects lying in $\proj(Q)$ over $\N$ in $\Rep(Q)$. Then the result follows from Lemma~\ref{lem:flat-proj}.
\end{proof}

The following result shows that the convex hull of any right infinite path in $\supp M$ is uniformly interval finite, if $M$ is a projective object in $\rep(Q)$.

\begin{lemma}\label{lem:not-proj}
  Let $M$ be a representation in $\rep(Q)$. Assume that $\supp M$ contains a right infinite path $p$ whose convex hull is not uniformly interval finite. Then $M$ is not projective in $\rep(Q)$.
\end{lemma}

\begin{proof}
  Assume $p = \alpha_1 \alpha_2 \cdots \alpha_i \cdots$. For each $i\geq 0$, we set $a_i = t(\alpha_{i+1})$.
  Assume that $M$ is projective in $\rep(Q)$.

  We observe that $\Hom(P_{a_i},M) \simeq M(a_i)$ and hence it is nonzero.
  Let $f_i\colon P_{a_i}\to M$ be a nonzero morphism. Since $M$ is projective in $\rep(Q)$ which is hereditary, we have that $f_i$ is an injection. Then
  \(
    \dim M(a_0)\geq \dim P_{a_i}(a_0) = \abs{Q(a_i, a_0)}
  \).

  Since $Q$ contains no oriented cycles, every right infinite path is irrational. Then the equivalence between Proposition~\ref{prop:uif}(1) and Proposition~\ref{prop:uif}(2) implies that $\set{\abs{Q(a_i, a_0)} \,\middle\vert i\geq 0}$ is unbounded. It follows that $M(a_0)$ is not finite dimensional, which is a contradiction. Then the result follows.
\end{proof}

We mention the following observations about projective objects in $\rep(Q)$.

\begin{lemma}\label{lem:M(alpha)}
  Let $M$ be a projective object in $\rep(Q)$. Then $M(\alpha)$ is an injection for any arrow $\alpha \colon a \to b$.
\end{lemma}

\begin{proof}
  We observe that there exists some morphism $f \colon P_a^{\oplus \dim M(a)} \to M$ such that $\Im f(a) = M(a)$. Since $\rep(Q)$ is hereditary and $M$ is projective, then $f$ is an injection. We observe that $P_a(\alpha)$ is an injection. Then the result follows.
\end{proof}

For any vertex $a$ and $x \in M(a)$, we denote by $M_x$ the subrepresentation of $M$ generated by $x$.

\begin{lemma}\label{lem:sum}
  Let $M$ be a projective object in $\rep(Q)$.
  Assume $a_1, a_2, \dots, a_n$ are vertices in $\supp M$, and $x_i \in M(a_i)$ is nonzero for each $1\leq i \leq n$ such that $x_i \notin \sum_{j \neq i} M_{x_j} (a_i)$.
  Then $\sum_{i=1}^n M_{x_i}$ is an internal direct sum.
\end{lemma}

%We mention that $\bigoplus_{i=1}^n M_{x_i}$ means internal direct sum.

\begin{proof}
  We mention that $M_{x_i} \simeq P_{a_i}$ for each $1 \leq i \leq n$. Consider the canonical epimorphism
  \[
    g \colon \bigoplus_{i=1}^n M_{x_i} \To \sum_{i=1}^n M_{x_i}.
  \]
  Here, the direct sum means external direct sum. We observe that $\sum_{i=1}^n M_{x_i}$ is projective in $\rep(Q)$, since $\rep(Q)$ is hereditary. Then $g$ is a split epimorphism.

  For each $1 \leq i \leq n$, we have that $( \rad M_{x_i} ) (a_i) = 0$. Then $( \rad\sum_{j=1}^n M_{x_j} ) (a_i)$ is contained in $\sum_{j \neq i} M_{x_j} (a_i)$.
  By the assumption, we have that $x_i$ does not lie in $( \rad\sum_{j=1}^n M_{x_j} ) (a_i)$.
  Then $\bigoplus_{i=1}^n S_{a_i}$ is a direct summand of $\top ( \sum_{i=1}^n M_{x_i} )$.
  We observe by the split epimorphism $g$ that $\top ( \sum_{i=1}^n M_{x_i} )$ is a direct summand of $\top ( \bigoplus_{i=1}^n M_{x_i} )$, which is isomorphic to $\bigoplus_{i=1}^n S_{a_i}$.
  It follows that
  \[
    \top \biggl( \sum_{i=1}^n M_{x_i} \biggr)
    \simeq \bigoplus_{i=1}^n S_{a_i}
    \simeq \top \biggl( \bigoplus_{i=1}^n M_{x_i} \biggr).
  \]
  Since both $\bigoplus_{i=1}^n M_{x_i}$ and $\sum_{i=1}^n M_{x_i}$ lie in $\proj(Q)$, then $g$ is an isomorphism. Then the result follows.
\end{proof}

\begin{proposition}\label{prop:proj}
  Let $M$ be an indecomposable projective object in $\rep(Q)$. Assume that $\supp M$ contains a right infinite path $p$. Then the convex hull of any right infinite path in $[p]$ is uniformly interval finite and $M\simeq X_{[p]}$.
\end{proposition}

\begin{proof}
  We observe by Lemma~\ref{lem:M(alpha)} that each right infinite path in $[p]$ is completely contained in $\supp M$. Then the convex hull of any right infinite path in $[p]$ is uniformly interval finite by Lemma~\ref{lem:not-proj}. It follows from Proposition~\ref{prop:Xp-proj} that $X_{[p]}$ is an indecomposable projective object in $\rep(Q)$.

  Assume $p = \alpha_1 \alpha_2 \cdots \alpha_i \cdots$. For each $i \geq 0$, we set $a_i = t(\alpha_{i+1})$. We have that $\dim M(a_i) \geq \dim M(a_{i+1})$ since $M(\alpha_{i+1})$ is an injection. Then $\dim M(a_i)$ will be stable for $i$ large enough. We may assume $\dim M(a_i) = \dim M(a_0)$ for each $i\geq 0$. Then $M(\alpha_i)$ is a bijection for each $i\geq 0$.

  Let $\Omega_0=\set{a_i|i\geq 0}$. We denote by $X_0$ the subrepresentation of $M$ generated by $\bigcup_{a\in\Omega_0} M(a)$.
  We observe that
  \[
    X_0 \simeq X_{[p]}^{\oplus \dim M(a_0)}.
  \]
  Let $\Omega_1 = \Omega_0 \cup Q_0 \setminus \supp X_0$.
  We denote by $X_1$ the subrepresentation of $M$ generated by $\bigcup_{a\in \Omega_1 \setminus \Omega_0} M(a)$.
  We observe that $X_0(a) + X_1(a) = M(a)$ for each $a\in\Omega_1$.

  We claim that $X_0 \cap X_1 = 0$.
  Indeed, let $x \in X_0(a) \cap X_1(a)$ for some vertex $a$.
  Then there exists pairwise different vertices $b_1, b_2, \dots, b_n$ in $\Omega_1 \setminus \Omega_0$ and $y_i \in X_1(b_i)$ for each $1 \leq i \leq n$ such that $x \in \sum_{i=1}^n M_{y_i} (a)$. We may assume $y_i \notin \sum_{j \neq i} M_{y_j} (b_i)$ for any $1 \leq i \leq n$.
  We observe that there exist some $j \geq 0$ such that there does not exist a finite path $u$ with $t(u) = a_j$ and $s(u) = b_i$ for any $1 \leq i \leq n$, since $Q$ is interval finite.
  We can choose $y_{n+1} \in M(a_j)$ such that $x \in M_{y_{n+1}} (a)$. We observe that $y_i \notin \sum_{j \neq i} M_{y_j} (b_i)$ for any $1 \leq i \leq n+1$.
  By Lemma~\ref{lem:sum}, we have that $x = 0$. It follows that $X_0 \cap X_1 = 0$.

  Let $\Omega_2 = \set{a \in Q_0 \middle\vert a^- \subseteq \Omega_1}$.
  For each $a \in \Omega_2 \setminus \Omega_1$,  we let $U_a$ be a complement to $X_0(a) \oplus X_1(a)$ in $M(a)$. Let $X_2$ be the subrepresentation of $M$ generated by $\bigcup_{a \in \Omega_2 \setminus \Omega_1} U_a$.
  We observe that $X_0(a) \oplus X_1(a) \oplus X_2(a) = M(a)$ for each $a \in \Omega_2$.

  We claim that $(X_0 \oplus X_1) \cap X_2=0$.
  Indeed, let $x \in (X_0 \oplus X_1)(a) \cap X_2(a)$ for some vertex $a$.
  Then there exist pairwise different vertices $b_1, b_2, \dots, b_n \in \Omega_2 \setminus \Omega_1$ and $y_i \in X_2 (b_i)$ for each $1 \leq i \leq n$ such that $x \in \sum_{i=1}^n M_{y_i} (a)$. We may assume $y_i \notin \sum_{1 \leq j \leq n, j \neq i} M_{y_j} (b_i)$ for any $1 \leq i \leq n$.
  Similarly, there exist pairwise different vertices $b_{n+1}, b_{n+2}, \dots, b_{n+m} \in \Omega_1$ and $y_i \in (X_0 \oplus X_1) (b_i)$ for each $n+1 \leq i \leq n+m$ such that $x \in \sum_{i=n+1}^{n+m} M_{y_i} (a)$. We may assume $y_i \notin \sum_{n+1 \leq j \leq n+m, j \neq i} M_{y_j} (b_i)$ for any $n+1 \leq i \leq n+m$.
  We observe that $y_i \notin \sum_{j \neq i} M_{y_j} (b_i)$ for any $1 \leq i \leq n+m$, since $\sum_{i=0}^2 X_i(a)$ is an internal direct sum for each $a \in \Omega_2$.
  By Lemma~\ref{lem:sum}, we have that $x = 0$. It follows that $(X_0 \oplus X_1) \cap X_2=0$.

  By induction, for each $n\geq0$, we have $\Omega_n$ and $X_n$ such that $\sum_{i=1}^n X_i$ is an internal direct sum and $(\bigoplus_{i=0}^n X_i) (a) = M(a)$ for each vertex $a\in\Omega_n$.

  We observe that any vertex in $\supp M$ lies in $\Omega_i$ for some $i\geq0$. It follows that $M=\bigoplus_{i=0}^\infty X_i$. Since $M$ is indecomposable, we have that $M=X_0\simeq X_{[p]}$.
\end{proof}

We then obtain the following classification theorem.

\begin{theorem}\label{thm:classify}
  Let $Q$ be an interval finite quiver. Assume that $P$ is an indecomposable projective object in $\rep(Q)$. Then either $P \simeq P_{a}$ for some vertex $a$, or $P \simeq X_{[p]}$ for some equivalence class $[p]$ of right infinite paths, where the convex hull of any right infinite path in $[p]$ is uniformly interval finite.
\end{theorem}

\begin{proof}
  If $\supp P$ contains a right infinite path, the result follows from Proposition~\ref{prop:proj}.
  Now, we assume that $\supp P$ contains no right infinite paths. Then the canonical epimorphism $P\to \top P$ is a projective cover by Lemma~\ref{lem:ess}(1). In particular, $\top P$ is nonzero. Let $S$ be a simple factor representation of $\top P$. We have an epimorphism $f\colon P\to S$. Assume that $S(a)\neq0$ for some vertex $a$. We have a projective cover $g\colon P_a\to S$.
  It follows that $P_a$ is a direct summand of $P$.
  Since $P$ is indecomposable, we have that $P\simeq P_a$.
\end{proof}

We mention that an indecomposable projective object in $\rep(Q)$ of the form $X_{[p]}$ is not projective in $\Rep(Q)$.
Indeed, if $p=\alpha_1 \alpha_2 \cdots \alpha_i \cdots$ and $a_i = t(\alpha_{i+1})$ for any $i \geq 0$, then the canonical epimorphism $\bigoplus_{i\geq0} P_{a_i} \to X_{[p]}$ is not a retraction.

By the duality $D \colon \rep(Q^\op) \to \rep(Q)$, we can also classify the injective objects in $\rep(Q)$; see the following theorem.
One can also obtain the classification of injective objects in $\rep(Q)$ analogous to the above process based on Lemmas~\ref{lem:proj-inj}(2) and \ref{lem:Yp-flat'}.

\begin{theorem}\label{thm:classify'}
  Let $Q$ be an interval finite quiver. Assume that $I$ is an indecomposable injective object in $\rep(Q)$. Then either $I \simeq I_{a}$ for some vertex $a$, or $I \simeq Y_{[p]}$ for some equivalence class $[p]$ of left infinite paths, where the convex hull of any left infinite path in $[p]$ is uniformly interval finite.
  \qed
\end{theorem}

We characterize the indecomposable projective objects and indecomposable injective objects in $\rep(Q)$ for the following infinite quiver of $A_\infty^\infty$ type.

\begin{example}
  Let $Q$ be the following quiver of $A_\infty^\infty$ type
  \[\begin{tikzcd}
    \cdots
    & \underset{-1}{\circ} \lar["\alpha_{-1}"']
    & \underset{ 0}{\circ} \lar["\alpha_0"']
    & \underset{ 1}{\circ} \lar["\alpha_1"']
    & \cdots. \lar["\alpha_2"']
  \end{tikzcd}\]

  For every vertex $i$, we denote by $P_i$ the corresponding indecomposable projective representation, and by $I_i$ the corresponding indecomposable injective representation.

  Set $p = \alpha_1 \alpha_2 \cdots \alpha_i \cdots$. Then $X_{[p]}$ is isomorphic to the representation
  \[\begin{tikzcd}
    \cdots
    & k \lar["1"']
    & k \lar["1"']
    & k \lar["1"']
    & \cdots. \lar["1"']
  \end{tikzcd}\]
  By Proposition~\ref{prop:Xp-proj}, we have that $X_{[p]}$ is an indecomposable projective object in $\rep(Q)$.
  By Theorem~\ref{thm:classify}, we have that
  \[
    \set{X_{[p]}} \cup \set{P_i \middle\vert i \in Q_0}
  \]
  is a complete set of indecomposable projective objects in  $\rep(Q)$.

  Set $q = \cdots \alpha_j \cdots \alpha_{-2} \alpha_{-1}$. Then $Y_{[q]} \simeq X_{[p]}$. By Theorem~\ref{thm:classify'}, we have that
  \[
    \set{Y_{[q]}} \cup \set{I_i \middle\vert i \in Q_0}
  \]
  is a complete set of indecomposable injective objects in $\rep(Q)$.
  \qed
\end{example}

It is worth mentioning that $\rep(Q)$ may not contain enough projective objects. Indeed, we have the following characterization.

\begin{proposition}\label{prop:proj-cover}
  The following statements are equivalent.
  \begin{enumerate}
    \item
      $\rep(Q)$ contains enough projective objects.
    \item
      Every object in $\rep(Q)$ admits a projective cover.
    \item
      Every vertex in $Q$ admits only finitely many predecessors.
  \end{enumerate}
%  $\rep(Q)$ contains enough projective objects if and only if every vertex in $Q$ admits only finitely many predecessors.
\end{proposition}

\begin{proof}
  (1) $\Rightarrow$ (3).
  Assume that a vertex $b$ admits infinitely many predecessors $a_i$ for $i \in \Lambda$. Then there are no projective objects $P$ in $\rep(Q)$ with some epimorphism $P \to \bigoplus_{i \in \Lambda} S_{a_i}$.
  Indeed, otherwise each $P_{a_i}$ is a direct summand of $P$. Then $P(b)$ is not finite dimensional, which is a contradiction.

  (3) $\Rightarrow$ (2).
  Assume every vertex in $Q$ admits only finitely many predecessors. In particular, $Q$ contains no right infinite paths.

  Let $M$ be a representation in $\rep(Q)$.
  Assume
  \(
    \top M \simeq \bigoplus_{i \in \Lambda} S_{a_i}
  \)
  for some index set $\Lambda$.
  Then we obtain the following commutative diagram
  \[\begin{tikzcd}
    \bigoplus_{i \in \Lambda} P_{a_i} & \bigoplus_{i \in \Lambda} S_{a_i}\\
    M                                & \top M.
    \ar[from=1-1, to=1-2, two heads, "g"]
    \ar[from=1-1, to=2-1, dashed, "h"']
    \ar[from=1-2, to=2-2, "\simeq"]
    \ar[from=2-1, to=2-2, two heads, "f"]
  \end{tikzcd}\]

  By the assumption, at most finitely many $a_i$ are predecessors of any given vertex $b$.
  Then $\bigoplus_{i \in \Lambda} P_{a_i} (b)$ is finite dimensional, since $Q$ is interval finite. That is to say, $\bigoplus_{i \in \Lambda} P_{a_i}$ lies in $\rep(Q)$.

  We observe that
  \(
    \rad \bigoplus_{i \in \Lambda} P_{a_i} \simeq \bigoplus_{i \in \Lambda} \rad P_{a_i}
  \).
  Since $Q$ contains no right infinite paths, then $f$ and $g$ are essential epimorphisms by Lemma~\ref{lem:ess}. It follows that $h$ is an epimorphism. Moreover, it is an essential epimorphism since so is $g$. Then it is a projective cover in $\rep(Q)$.

  (2) $\Rightarrow$ (1).
  This is straightforward.
\end{proof}

We mention that, if the equivalent conditions in the preceding proposition hold, there are no indecomposable projective objects of the form $X_{[p]}$ in $\rep(Q)$.

Similarly, one can prove that $\rep(Q)$ admits enough injective objects if and only if every vertex in $Q$ admits only finitely many successors.

\section{A characterization for projective objects in \texorpdfstring{$\rep(Q)$}{rep(Q)}}
\label{sec:flat}

Let $Q$ be an interval finite quiver.
We strengthen Proposition~\ref{prop:flat} for countably generated flat representations in $\Rep(Q)$ as follows.

\begin{lemma}\label{lem:countable}
  A countably generated flat representation in $\Rep(Q)$ is a direct limit of finitely generated projective subrepresentations over the directed set $(\N, \leq)$.
\end{lemma}

\begin{proof}
  Assume that $M$ is a countably generated flat representation.
  By Proposition~\ref{prop:flat}, we may assume that $M$ is a direct limit of subrepresentations $P_\lambda$ lying in $\proj(Q)$ over some directed set $(\Lambda, \leq)$.
  Since $M$ is countably generated, there exist some vertices $a_i$ for $i \geq 0$ and some $x_i \in M(a_i)$ such that $M = \sum_{i \geq 0} M_{x_i}$.

  We observe that for any $i \geq 0$, there exists some $\lambda_i \in \Lambda$ such that $x_i \in P_{\lambda_i} (a_i)$. Then $M_{x_i} \subseteq P_{\lambda_i}$.
  Let $\lambda'_0 = \lambda_0$, and choose some $\lambda'_{i+1}$ with $\lambda'_{i+1} \geq \lambda'_i, \lambda_{i+1}$ for any $i \geq 0$ inductively.
  Then for any $1 \leq j \leq i$, we have that $x_j \in P_{\lambda'_i} (a_j)$ and hence $M_{x_j} \subseteq P_{\lambda'_j} \subseteq P_{\lambda'_i}$.
  We observe that $(P_{\lambda'_i}, \subseteq)$ forms a direct system over the directed set $(\N, \leq)$. Then $M = \sum_{i \geq 0} M_{x_i} = \bigcup_{i \geq 0} P_{\lambda'_i}$. It implies that $M = \varinjlim P_{\lambda'_i}$. Then the result follows since $P_{\lambda'_i} \in \proj(Q)$ for any $i \geq 0$.
\end{proof}

Then we obtain the following relationship between projective objects in $\rep(Q)$ and flat representations lying in $\rep(Q)$.
This is analogous to the classical result \cite[Theorem~2.2]{Drinfeld2006Infinite} in module categories, which is due to \cite{Kaplansky1958Projective} and \cite{RaynaudGruson1971Criteres}.

\begin{theorem}\label{thm:flat}
  Let $Q$ be an interval finite quiver and let $M$ be a representation in $\rep(Q)$. Then $M$ is projective in $\rep(Q)$ if and only if $M$ is a direct sum of countably generated flat representations in $\Rep(Q)$.
\end{theorem}

\begin{proof}
  For the sufficiency, we assume $M = \bigoplus_{i \in \Lambda} M_i$ for some indexed set $\Lambda$, such that each $M_i$ is a countably generated flat representation in $\Rep(Q)$. Lemma~\ref{lem:countable} implies that $M_i$ is a direct limit of representations lying in $\proj(Q)$ over $\N$. It follows from Lemma~\ref{lem:flat-proj} that $M_i$ is projective in $\rep(Q)$. Then so is their direct sum, i.e., $M$ is projective in $\rep(Q)$.

  For the necessity, we assume $M$ is projective in $\rep(Q)$.
  By Lemma~\ref{lem:dec}, we have a decomposition $M = \bigoplus_{i \in \Lambda} M_i$ for some indexed set $\Lambda$, such that each $M_i$ is indecomposable.
  Then each $M_i$ is a projective object in $\rep(Q)$.
  Theorem~\ref{thm:classify} implies that either $M_i \simeq P_a$ for some vertex $a$ or $M_i \simeq X_{[p]}$ for some equivalence class $[p]$ of right infinite paths. Then the result follows since both $P_a$ and $X_{[p]}$ are countably generated flat representations in $\Rep(Q)$.
\end{proof}

\begin{corollary}
  Let $Q$ be a connected strongly locally finite quiver and let $M$ be a representation in $\rep(Q)$. Then $M$ is projective in $\rep(Q)$ if and only if $M$ is a flat representation in $\Rep(Q)$.
\end{corollary}

\begin{proof}
  We observe that a connected strongly locally finite quiver is countable. Then each pointwise finite dimensional representation of $Q$ is countably generated. We mention that flat representations in $\Rep(Q)$ are closed under direct sums. Then the result is a direct consequence of Theorem~\ref{thm:flat}.
\end{proof}

\section*{Acknowledgements}

The author is very grateful to Professor~Xiao-Wu Chen for his encouragement and many helpful suggestions.
%Proposition~\ref{prop:inf-sum} is due to his suggestion.
The author also thanks Professor~Yu Ye and Professor~Shiping Liu for some suggestions, and thanks Doctor~Zhe Han, Doctor~Bo Hou and Doctor~Dawei Shen for some discussions.
He would like to thank the referee for many helpful suggestions and comments.

% ----------------------------------------------------------------
%\bibliographystyle{siam}
%\bibliography{../bib/math,../bib/preprint}

\end{document}